\documentclass{amsart}

\newtheorem{theorem}{Theorem}[section]
\newtheorem{lemma}[theorem]{Lemma}
\newtheorem{corollary}[theorem]{Corollary}

\theoremstyle{definition}

\theoremstyle{remark}

\numberwithin{equation}{section}

\begin{document}

\title{ Miyaoka-Yau
  inequality for minimal projective manifolds  of general type }

\author{Yuguang Zhang}
\address{Department of Mathematics, Capital Normal University,
Beijing, P.R.China }
\curraddr{Department of Mathematical Sciences, Korea Advanced
Institute of Science and Technology, Daejeon, Republic of  Korea}
\email{zhangyuguang76@yahoo.com}
%

\dedicatory{}

\begin{abstract}
In this short
 note, we  prove  the Miyaoka-Yau
  inequality  for minimal projective $n$-manifolds of
general type by
 using K\"ahler-Ricci flow.
\end{abstract}

\maketitle



\section{Introduction}
If  $M$ is   a  projective  $n$-manifold  with ample canonical
bundle $\mathcal{K}_{M}$, there exists a  K\"ahler-Einstein metric
$\omega$ with negative scalar curvature by Yau's
  theorem    on  the  Calabi  conjecture (\cite{Ya2}), which was obtained by Aubin independently (\cite{Au}).   As a
  consequence, there is an   inequality for Chern numbers,  the Miyaoka-Yau
  inequality,    \begin{equation}\label{1.1}(\frac{2(n+1)}{n}c_{2}(M)-c_{1}^{2}(M))\cdot (-c_{1}(M))^{n-2}\geq 0,
  \end{equation} where $ c_{1}(M)$ and $c_{2}(M)$ are the first and
  the second Chern classes  of $M$ (c.f. \cite{Ya1}). Furthermore, if
  the equality in (\ref{1.1}) holds, the K\"ahler-Einstein metric
$\omega$ is a complex hyperbolic metric, i.e. the holomorphic
 sectional curvature of $\omega$ is a negative constant.  If $n=2$,
(\ref{1.1}) even holds for algebraic surfaces of general type (c.f.
\cite{CY}, \cite{Kob}, \cite{M}), which may not admit any
K\"ahler-Einstein metric.   In \cite{Ts2}, the inequality
(\ref{1.1}) is proved  for any dimensional minimal projective
manifold of general type by using singular K\"ahler-Einstein
metrics.  In this short note, we give a different proof of
(\ref{1.1}) for minimal projective $n$-manifolds of general type by
 using K\"ahler-Ricci flow, and study the extremal  case of
 (\ref{1.1}).

Let $M$ be  a minimal projective  manifold of general type with
$dim_{\mathbb{C}}M=n\geq 2$. The canonical bundle $\mathcal{K}_{M}$
of $M$  is  big,    and semi-ample, i.e. $ \mathcal{K}_{M}^{n}>0$,
and,  for a positive integer $m\gg 1$, the  linear system
$|m\mathcal{K}_{M}|$ is base point free (as quoted in  \cite{Ts}).
For $m\gg 1$, the complete linear system $|m\mathcal{K}_{M}|$
defines a holomorphic  map $\Phi: M \longrightarrow
\mathbb{CP}^{N}$, which is birational onto its image $M_{can}$.
$M_{can}$ is called the canonical model of $M$, and $\Phi$ is called
  the contraction map. Note that $M$ may not admit any
  K\"ahler-Einstein metric.  The K\"ahler-Ricci flow is an  evolution
  equation of a family of K\"ahler metrics $\omega_{t}$, $t\in
[0,T)$,  on $M$, \begin{equation}\label{1.01}
\partial_{t}\omega_{t}=-Ric(\omega_{t})-\omega_{t},
\end{equation}  where $Ric(\omega_{t})$ is the Ricci form of
$\omega_{t}$. By \cite{Ts} \cite{TZ}
 \cite{CL} and \cite{Z},  for any  K\"ahler metric as
  initial metric, the solution $\omega_{t}$ of the K\"ahler-Ricci flow
  equation   exists for
 all time $t\in [0, \infty)$, and the scalar curvature of
 $\omega_{t}$ is  uniformly bounded.  Thus   we can prove (\ref{1.1}) by   using
    the technique developed  in
 \cite{FZZ}, where a Hitchin-Thorpe type inequality was proved  for
 $4$-manifolds  which admit  a long time solution to  a  normalized  Ricci flow
 equation with bounded scalar curvature. Before proving  the Miyaoka-Yau
  inequality, we show that the $L^{2}$-norm of the  Einstein tensor
  tends to zero along a subsequence of a solution of the K\"ahler-Ricci flow
  equation  (\ref{1.01}).

\vskip 5mm

\begin{theorem} Let  $M$ be  a minimal projective  manifold of general type
with $dim_{\mathbb{C}}M=n\geq 2$, and $\omega_{t}$, $t\in [0,
\infty)$, be a solution of the K\"ahler-Ricci flow
  equation  (\ref{1.01}).  Then there exists a sequence of times
  $t_{k}\longrightarrow \infty$, when $k\longrightarrow\infty$, such
  that
$$\lim_{k\longrightarrow \infty}
\int_{M}|\rho_{t_{k}}|^{2}\omega_{t_{k}}^{n}=0,
$$ where $\rho_{t_{k}}=Ric_{t_{k}}-\frac{R_{t_{k}}}{n}\omega_{t_{k}}$ denotes
the Einstein tensor of $\omega_{t_{k}}$, and  $R_{t_{k}}$ denotes
the scalar curvature of $\omega_{t_{k}}$.
\end{theorem}

\vskip 5mm
 As a corollary of this theorem, we obtain the   Miyaoka-Yau
  inequality for minimal projective  manifolds  of general type.
\vskip 5mm

\begin{corollary}
If $M$ is a minimal projective  manifold of general type with
$dim_{\mathbb{C}}M=n\geq 2$, then
$$(\frac{2(n+1)}{n}c_{2}(M)-c_{1}^{2}(M))\cdot (-c_{1}(M))^{n-2}\geq 0.
$$ Furthermore, if the equality  holds,  there is a complex hyperbolic
metric on the smooth part $M_{0}$ of the canonical model $M_{can}$
of $M$.
\end{corollary}

\vskip 5mm

\noindent {\bf Acknowledgement:}  The  author   thanks  Zhou Zhang
for explaining  \cite{Z}, and sending him  the paper \cite{TZ}.
Thanks also goes to referees  for their suggestions  of  improving
the present paper. The  author  also  thanks Valentino Tosatti for
sending him \cite{Ts2}.

\section{Proof of Theorem 1.1}
Let $M$ be a minimal projective manifold of general type with
$dim_{\mathbb{C}}M=n\geq 2$, $M_{can}$ be the canonical model of
$M$, and $\Phi: M\longrightarrow M_{can}$ be the contraction map.
  Consider the   K\"ahler-Ricci flow
  equation  on $M$, \begin{equation}\label{2.1}
\partial_{t}\omega_{t}=-Ric(\omega_{t})-\omega_{t},
\end{equation}  with initial metric $\omega_{0}$.  In \cite{H},
the short time existence of the solution of (\ref{2.1}) is proved.
Then,  in \cite{Ts} \cite{TZ} and \cite{CL}, it is proved that the
solution $\omega_{t}$ of (\ref{2.1}) exists for all time, i.e. $t\in
[0, +\infty)$, and, there exists a unique semi-positive current
$\omega_{\infty}$ on $M$, which   satisfies that
\begin{enumerate}
\item $\omega_{\infty}$ represents $-2\pi c_{1}(M)$.
\item $\omega_{\infty}$ is a smooth  K\"ahler-Einstein metric
with negative scalar curvature on $\Phi^{-1}(M_{0})$, where $M_{0}$
is the smooth part of $M_{can}$.
\item On any compact subset $K\subset \Phi^{-1}(M_{0})$,
$\omega_{t}$ $C^{\infty}$-converges to $\omega_{\infty}$ when
$t\longrightarrow\infty$.
\end{enumerate}
 In \cite{Z},  it is shown that there is a constant $C>0$ depending
 only on $\omega_{0}$ such that \begin{equation}\label{2.2}|R_{t}|<C,
 \end{equation} where $R_{t}$ is the scalar curvature of
 $\omega_{t}$.

  First,  we need
evolution equations for  volume forms and scalar curvatures as
follows,
\begin{equation}\label{2.3}\partial_{t}\omega_{t}^{n}=-(R_{t}+n)\omega_{t}^{n}, \ \ \ \ {\rm
and}\end{equation}
\begin{equation}\label{2.4}\partial_{t}R_{t}=\triangle_{t}
R_{t}+|Ric_{t}|^{2}+R_{t}= \triangle_{t}
R_{t}+|Ric_{t}\textordmasculine|^{2}-(R_{t}+n),\end{equation} where
$Ric_{t}\textordmasculine=Ric_{t}+\omega_{t}$, and
$|Ric_{t}\textordmasculine|^{2}=|Ric_{t}|^{2}+2R_{t}+n $ (c.f. Lemma
2.38 in \cite{C-N}).

\begin{lemma} There  are two  constants $t_{0}>0$  and  $c>0$
independent of
  $t$  such
that, for $t> t_{0}$,
$$\breve{R}_{t}=\inf_{x\in M}R_{t}(x)\leq-n+e^{-t}c< -\frac{n}{2}<0.$$
\end{lemma}
 \begin{proof}
If we define $\alpha_{t}=[\omega_{t}] \in H^{1,1}(M, \mathbb{R})$,
 from (\ref{2.1}) we have
$$
\partial_{t}\alpha_{t}=-2\pi c_{1}(M)-\alpha_{t}, \ \ \ {\rm and} \ \ \ $$ \begin{equation}\label{2.20}\alpha_{t}= -2\pi c_{1}(M)+e^{-t}(2\pi c_{1}(M)+\alpha_{0}).
\end{equation} Thus
\begin{equation}\label{2.30}[\omega_{\infty}]=\alpha_{\infty}=\lim_{t\longrightarrow\infty}
\alpha_{t}=-2\pi c_{1}(M).\end{equation}  Since
$$\breve{R}_{t}\int_{M}\omega_{t}^{n}\leq
 \int_{M}R_{t}\omega_{t}^{n}=n\int_{M}Ric_{t}\wedge
 \omega_{t}^{n-1}=n2\pi c_{1}(M)\cdot \alpha_{t}^{n-1},$$ we obtain $$\breve{R}_{t}\leq n\frac{2\pi c_{1}(M)\cdot \alpha_{t}^{n-1}}{\alpha_{t}^{n}}
 =n\frac{2\pi c_{1}(M)\cdot \alpha_{t}^{n-1}}{-2\pi c_{1}(M)\cdot
 \alpha_{t}^{n-1}+e^{-t}(2\pi c_{1}(M)+\alpha_{0})\cdot
 \alpha_{t}^{n-1}}= \frac{-n}{1+e^{-t}A_{t}}, $$ where $A_{t}=-\frac{(2\pi c_{1}(M)+\alpha_{0})\cdot
 \alpha_{t}^{n-1}}{2\pi c_{1}(M)\cdot
 \alpha_{t}^{n-1}}$. Note that $
 (-c_{1}(M))^{n}>0$.   Thus there is a $t_{1}>0$ such that, if
 $t>t_{1}$, $A_{t}<|\frac{(\alpha_{\infty}+\alpha_{0})\cdot
 \alpha_{\infty}^{n-1}}{
 \alpha_{\infty}^{n}}|+1=A$, and  we obtain that $$\breve{R}_{t}\leq
 \frac{-n}{1+e^{-t}A}<-n+e^{-t}c,$$ where
 $c=-n(\frac{A}{1+e^{-t_{1}}A})$. By taking $t_{0}>t_{1}$ such that
 $e^{-t_{0}}c< \frac{n}{2} $, we obtain the conclusion.
  \end{proof}

  \begin{lemma}
$$
\  \int_{0}^{\infty}\int_{M}|R_{t}+n|\omega_{t}^{n}dt<\infty.$$
\end{lemma}

\begin{proof}

By (\ref{2.4})  and the   maximal principle,
$\partial_{t}\breve{R}_{t} \geq -(\breve{R}_{t}+n),$ and  so,
\begin{equation}\label{2.7}n+\breve{R}_{t}\geq
Ce^{-t},\end{equation} for a  constant $ C$ independent of $t$.
  Note that, by Lemma 2.1, (\ref{2.7}) and (\ref{2.20}), when
  $t>t_{0}$,
\begin{eqnarray*}\int_{M}|R_{t}+n|\omega_{t}^{n}& \leq & \int_{M}(R_{t}-\breve{R}_{t})\omega_{t}^{n}+\int_{M}|n+\breve{R}_{t}|\omega_{t}^{n}
\\ & \leq & \int_{M}(R_{t}+n)\omega_{t}^{n}+2\int_{M}|n+\breve{R}_{t}|\omega_{t}^{n}
\\   & \leq &
\int_{M}(R_{t}+n)\omega_{t}^{n}+C_{3}e^{-t}\\ & = & n(2\pi
c_{1}\cdot \alpha_{t}^{n-1}+\alpha_{t}^{n})+C_{3}e^{-t}\\ & =&
ne^{-t}(2\pi c_{1}+\alpha_{0})\cdot \alpha_{t}^{n-1}+C_{3}e^{-t}
\\ & \leq & C_{4}e^{-t},\end{eqnarray*} for two constants $ C_{3}$ and
$ C_{4}$ independent of $t$.  Thus
$$\int_{0}^{\infty}\int_{M}|R_{t}+n|\omega_{t}^{n}dt=\int_{0}^{t_{0}}\int_{M}|R_{t}+n|\omega_{t}^{n}dt+\int_{t_{0}}^{\infty}\int_{M}|R_{t}+n|\omega_{t}^{n}dt <\infty.$$
\end{proof}

\begin{proof}[Proof of Theorem 1.1]
  From  (\ref{2.4}), (\ref{2.3}),  (\ref{2.30}), (\ref{2.2}),  and  Lemma 2.2,  we obtain
\begin{eqnarray}
\int_{0}^{\infty}\int_{M}
|Ric\textordmasculine_{t}|^{2}\omega_{t}^{n}
dt&=&\int_{0}^{\infty}\int_{M}(\frac{\partial}{\partial t}R_{t})
\omega_{t}^{n} dt+\int_{0}^{\infty}\int_{M}(R_{t}+n)\omega_{t}^{n} dt\nonumber\\
&=&\int_{0}^{\infty}\frac{\partial}{\partial
t}(\int_{M}R_{t}\omega_{t}^{n}) dt+
\int_{0}^{\infty}\int_{M} ( R_{t}+1)(R_{t}+n)\omega_{t}^{n} dt\nonumber\\
&\leq& n\alpha_{\infty}^{n}-\int_{M}R_{0}\omega_{0}^{n}+C\int_{0}^{\infty}\int_{M}|R_{t}+n|\omega_{t}^{n} dt\nonumber\\
&< & \infty.\nonumber
\end{eqnarray} If $\rho_{t}=Ric_{t}-\frac{R_{t}}{n}\omega_{t}$ is
the Einstein tensor of $\omega_{t}$, then
$|\rho_{t}|^{2}=|Ric\textordmasculine_{t}|^{2}-\frac{1}{n}(R_{t}+n)^{2}$,
and, from the above estimation, $$\int_{0}^{\infty}\int_{M}
|\rho_{t}|^{2}\omega_{t}^{n} dt\leq \int_{0}^{\infty}\int_{M}
|Ric\textordmasculine_{t}|^{2}\omega_{t}^{n} dt< \infty.$$ Thus
there is a sequence $t_{k}\longrightarrow \infty$ such that
$$\lim_{k\longrightarrow \infty} \int_{M}|\rho
_{t_{k}}|^{2}\omega_{t_{k}}^{n}=0.$$
\end{proof}

\begin{proof}[Proof of Corollary  1.2]
Note that the K\"ahler curvature tensor has a decomposition
$$Rm_{t}=\frac{R_{t}}{2n^{2}}\omega_{t}\otimes
\omega_{t}+\frac{1}{n}\omega_{t}\otimes \rho_{t} + \frac{1}{n}
\rho_{t} \otimes \omega_{t} + B_{t} $$ (c.f. (2.63) and (2.38)   in
\cite{Be}). By Chern-Weil theory,
$$(\frac{2(n+1)}{n}c_{2}(M)-c_{1}^{2}(M))\cdot
[\omega_{t}]^{n-2}=\frac{(n-2)!}{4\pi^{2}n!}\int_{M}(\frac{n+1}{n}|B_{0,t}|^{2}-\frac{(n^{2}-2)}{n^{2}}|\rho_{t}
|^{2})\omega_{t}^{n}
$$ (c.f. (2.82a) and (2.67) in
\cite{Be}),  where $B_{0,t}=B_{t}-\frac{{\rm tr} B_{t}}{n^{2}-1}{\rm
Id}$ is the tensor given by (2.64) in \cite{Be} corresponding to
$\omega_{t}$.
 By Theorem 1.1,
there is a sequence $t_{k}\longrightarrow \infty$ such that
$$\lim_{k\longrightarrow \infty} \int_{M}|\rho
_{t_{k}}|^{2}\omega_{t_{k}}^{n}=0.$$ Hence
\begin{eqnarray*}(\frac{2(n+1)}{n}c_{2}(M)-c_{1}^{2}(M))\cdot
(-2\pi c_{1}(M))^{n-2}& = &
(\frac{2(n+1)}{n}c_{2}(M)-c_{1}^{2}(M))\cdot
[\omega_{\infty}]^{n-2}\\ &=&  \lim_{k\longrightarrow
\infty}(\frac{2(n+1)}{n}c_{2}(M)-c_{1}^{2}(M))\cdot
[\omega_{t_{k}}]^{n-2}\\ &=&  \lim_{k\longrightarrow \infty}
\frac{(n-2)!}{4\pi^{2}n!}\int_{M}(\frac{n+1}{n}|B_{0,t_{k}}|^{2})\omega_{t_{k}}^{n}\\
&\geq & 0.\end{eqnarray*}

 If the  equality  holds, on any compact
subset $K\subset \Phi^{-1}(M_{0})$,
$$\int_{K}|B_{0,\infty}|^{2}\omega_{\infty}^{n}\leq
\lim_{k\longrightarrow \infty}
\int_{M}|B_{0,t_{k}}|^{2}\omega_{t_{k}}^{n}=0,$$ by the smooth
convergence of $\omega_{t}$ to $\omega_{\infty}$. Thus
$B_{0,\infty}\equiv 0$.  Since $\omega_{\infty}$ is a
K\"ahler-Einstein metric with negative scalar curvature on
$\Phi^{-1}(M_{0})$, the holomorphic sectional curvature is a
negative constant  by Section 2.66 in \cite{Be}, i.e.
$\omega_{\infty}$ is a complex hyperbolic metric.
\end{proof}

\bibliographystyle{amsplain}

\end{document}